\documentclass{amsart}
\usepackage{amssymb}
\usepackage{graphicx}

\numberwithin{equation}{section}
\newtheorem{Def}{Definition}
\newtheorem{Thm}{Theorem}
\newtheorem{Rmk}{Remark}
\newtheorem{Cor}{Corollary}
\newtheorem{Lem}{Lemma}

\newtheorem{Ex}{Example}
\newtheorem{Obs}{Observation}

\date{}
\begin{document}

\title[Curvature and the spectrum of the graph Laplace operator]{Ollivier-Ricci curvature and the spectrum of the normalized graph Laplace operator}
\author{Frank Bauer}
 \email{Frank.Bauer@mis.mpg.de}
\address{Max Planck Institute for Mathematics in the Sciences\\
04103 Leipzig, Germany.}

\author[1,2,3]{J\"urgen Jost}
 \email{jost@mis.mpg.de}
\address{Max Planck Institute for Mathematics in the Sciences\\
04103 Leipzig, Germany.}
\address{Department of Mathematics and Computer Science \\University of
Leipzig \\04109 Leipzig, Germany }
\address{Santa Fe Institute for the Sciences of Complexity, Santa Fe,
NM 87501, USA}
\author[1,4]{Shiping Liu}
\email{shiping@mis.mpg.de}
\address{Max Planck Institute for Mathematics in the Sciences\\
04103 Leipzig, Germany.}
\address{Academy of Mathematics and
Systems Science, Chinese Academy of Sciences, Beijing 100190,
China.}
\begin{abstract} We prove the following estimate for the spectrum of the normalized Laplace operator $\Delta$ on a finite graph $G$,
\begin{equation*}
1- (1- k[t])^{\frac{1}{t}}\leq \lambda_1 \leq \cdots \leq
\lambda_{N-1}\leq 1+ (1- k[t])^{\frac{1}{t}}, \,\forall
\,\,\text{integers}\,\, t\geq 1.
\end{equation*}
Here $k[t]$ is a lower bound for the Ollivier-Ricci curvature on
the neighborhood graph $G[t]$, which was introduced by Bauer-Jost.
In particular, when $t=1$ this is Ollivier's estimates $k\leq
\lambda_1\leq \ldots \leq \lambda_{N-1}\leq 2-k$. For sufficiently
large $t$ we show that, unless $G$ is bipartite, our estimates for
$\lambda_1$ and $\lambda_{N-1}$ are always nontrivial and improve
Ollivier's estimates for all graphs with $k\leq 0$. By definition
neighborhood graphs are weighted graphs which may have loops. To
understand the Ollivier-Ricci curvature on neighborhood graphs, we
generalize a sharp estimate of the Ricci curvature given by
Jost-Liu to weighted graphs with loops and relate it to the
relative local frequency
of triangles and loops. 
\end{abstract}
\maketitle
\section{Introduction}
In this paper, we utilize techniques inspired by Riemannian
geometry and the theory of stochastic processes in order to
control eigenvalues of graphs. In particular, we shall quantify
the deviation of a (connected, undirected, weighted, finite) graph
$G$ from being bipartite (a bipartite graph is one without cycles
of odd lengths; equivalently, its vertex set can be split into two
classes such that edges can be present only between vertices from
different classes) in terms of a spectral gap. The operator whose
spectrum we shall consider here is the normalized graph Laplacian
$\Delta$. This is the operator underlying random walks on graphs,
and so, this leads to a natural connection with the theory of
stochastic processes. We observe that on a bipartite graph, a
random walker, starting at a vertex $x$ at time 0 and at each step
hopping to one of the neighbors of the vertex where it currently
sits, can revisit $x$ only at even times. This connection then
will be explored via the eigenvalues of $\Delta$. More precisely,
the largest eigenvalue $\lambda_{N-1}$ of $\Delta$ is 2 iff $G$ is
bipartite and is $<2$ else. Therefore, $2-\lambda_{N-1}$
quantifies the deviation of $G$ from being bipartite, and we want
to understand this aspect in more detail. In more general terms,
we are asking for a quantitative connection between the geometry
(of the graph $G$) and the analysis (of the operator $\Delta$, or
the random walk encoded by it). Now, such connections have been
explored systematically in Riemannian geometry, and many
eigenvalue estimates are known there that connect the
corresponding Laplace operator with the geometry of the underlying
space $M$, see e.g. Li-Yau \cite{LiYau}, Chavel \cite{Chavel}. The
crucial role here is played by the Ricci curvature of $M$. In
recent years, a kind of axiomatic approach to curvature has been
developed. This approach encodes the abstract formal properties of
curvature and thereby makes the notion extendible to spaces more
general than Riemannian manifolds. By now, there exist many
notions of generalized curvature, and several of them have found
important applications, see Sturm \cite{Sturm1}, Lott-Villani
\cite{LV}, Ollivier \cite{Oll}, Ohta \cite{Ohta}, Bonciocat-Sturm
\cite{BS}, Joulin-Ollivier \cite{JO} and the references therein.
The curvature notion that turns out to be most useful for our
purposes is the one introduced by Ollivier \cite{Oll}. In his
paper, Olliver actually showed that the eigenvalues of the
normalized Laplace operator satisfy
\begin{equation}
\label{olli} k\leq \lambda_1\leq \ldots\leq \lambda_{N-1}\leq 2-k.
\end{equation}
In fact, one of the main points of the present paper is to relate
lower bounds for $\lambda_1$ and upper bounds for $\lambda_{N-1}$
via random walks. As in Bauer-Jost \cite{BJ}, we translate this
relationship into the geometric concept of a neighborhood graph.
The idea here is that in the $t$-th neighborhood graph $G[t]$ of
$G$, vertices $x$ and $y$ are connected by an edge with a weight
given by the probability that a random walker starting at $x$
reaches $y$ after $t$ steps times the degree of $x$. We note that
even though the original graph may have been unweighted, the
neighborhood graphs $G[t]$ are necessarily weighted. In addition,
they will in general possess self-loops, because the random walker
starting at $x$ may return to $x$ after $t$ steps. Therefore, we
need to develop our theory on weighted graphs with self-loops even
though the original $G$ might have been unweighted and without
such loops. Since Ollivier's curvature is defined in terms of
transportation distances (Wasserstein metrics), we can then use
our neighborhood graphs in order to geometrically control the
transportation costs and thereby to estimate the curvature of the
neighborhood graphs in terms of the curvature of the original
graph. As it turns out that lower bounds for the smallest
eigenvalue of $G[t]$, $t$ even, are related to upper bounds for
the largest eigenvalue of $G$, we obtain the following more
general estimate
\begin{equation}
\label{ours} 1- (1- k[t])^{\frac{1}{t}}\leq \lambda_1 \leq \cdots
\leq \lambda_{N-1}\leq 1+ (1- k[t])^{\frac{1}{t}}, \,\forall
\,\,\text{integers}\,\, t\geq 1.
\end{equation}
Whereas (\ref{olli}) is only useful under the restrictive
assumption that $k$ be positive, our estimate (\ref{ours}) is
nontrivial for any graph that is not bipartite. In fact, for a non-bipartite graph, we obtain an exponential decay of $1-k[t]$ with a rate that can be controlled by the geometry of the graph.

For controlling the smallest eigenvalue, besides Ollivier
\cite{Oll}, we also refer to Lin-Yau \cite{LinYau} and Jost-Liu
\cite{JL}. In particular, in the last paper, we could relate
$\lambda_1$ to the local clustering coefficient introduced in
Watts-Strogatz \cite{WS}. The local clustering coefficients
measures the relative local frequency of triangles, that is,
cycles of length 3. Since bipartite graphs cannot possess any
triangles, this then is obviously related to our question about
quantifying the deviation of the given graph $G$ from being
bipartite. In fact, in Jost-Liu \cite{JL}, this local clustering
has been controlled in terms of Ollivier's Ricci curvature. Thus,
in the present paper we are closing the loop between the geometric
properties of a graph $G$, the spectrum of its graph Laplacian,
random walks on $G$, and the generalized curvature of $G$, drawing
upon deep ideas and concepts originally developed in Riemannian
geometry and the theory of stochastic processes.

\section{The normalized Laplace operator, neighborhood graphs, and Ollivier-Ricci curvature}
In this paper, $G=(V, E)$ will denote an undirected, weighted,
connected, finite graph of $N$ vertices. We do not exclude loops,
i.e., we permit the existence of an edge between a vertex and
itself. $V$ denotes the set of vertices and $E$ denotes the set of
edges. If two vertices $x, y\in V$ are connected by an edge, we
say $x$ and $y$ are neighbors, in symbols $x\sim y$. The
associated weight function $w$: $V\times V\rightarrow \mathbb{R}$
satisfies $w_{xy}=w_{yx}$ (because the graph is undirected)  and
we assume $w_{xy}>0$ whenever $x\sim y$ and $w_{xy}=0$ iff
$x\not\sim y$. For a vertex $x\in V$, its degree $d_x$ is defined
as $d_x:=\sum_{y\in V}w_{xy}$. If $w_{xy}=1$ whenever $x\sim y$,
we shall call the graph an unweighted one. We will also consider a
locally finite graph $\tilde{G}=(\tilde{V}, \tilde{E})$, which is
an undirected, weighted, connected graph with a possible infinite
number of vertices that satisfies the property that for every
$x\in \tilde{V}$, the number of edges connected to $x$ is finite.

\subsection{The normalized graph Laplace operator and its eigenvalues}
In this subsection, we recall the definition of the normalized
graph Laplace operator and state some of its basic properties. In
particular, we will emphasize the relations between eigenvalues of
the Laplace operator and random walks on graphs.

Let $C(V)$ denote the space of all real-valued functions on the
set $V$ and let $m_x(\cdot)$ be a probability measure attached to
a vertex $x\in V$.
\begin{Def}
The Laplace operator $\Delta:C(V)\to C(V)$ is pointwise defined by
\begin{equation}\label{1}
\Delta f(x)=\sum_{y\in V}f(y)m_x(y)-f(x),\,\,\,\forall x\in V.
\end{equation}
\end{Def}
The measure $m_x(\cdot)$ can also be considered as the
distribution of a $1$-step random walk starting from $x$. We will
choose
\begin{equation}\label{2}
 m_x(y)=\left\{
          \begin{array}{ll}
            \frac{w_{xy}}{d_x}, & \hbox{if $y \sim x$,} \\
            0, & \hbox{otherwise,}
          \end{array}
        \right.
\end{equation}
in the following. Note that $x\sim x$ is possible when $x$ has a
loop. On a graph $G$ without loops, we can also consider a lazy
random walk. A lazy random walk is a random walk that does not
move with a certain probability, i.e. for some $x$ we might have
$m_x(x)\neq 0$. In this case, the lazy random walk on $G$ is
equivalent to the usual random walk on the graph
$G^{\mathrm{lazy}}$ that is obtained from $G$ by adding for every
vertex $x$ a loop with the weight $(d_x+m_x(x)) m_x(x)$, where
$d_x$ is the degree of $x$ in $G$.

With the family \eqref{2} of probability measures
$\{m_x(\cdot)\}$, $\Delta$ is just the normalized graph Laplace
operator studied for instance in \cite{Grigoryan,BJ} and is
unitarily equivalent to the Laplace operator studied in
\cite{Chung}.

We also have a natural measure $\mu$ on the whole set $V$,
$\mu(x):=d_x,$ which gives an inner product structure on $C(V)$.
\begin{Def}
The inner product of two functions $f, g\in C(V)$ is defined as
\begin{equation}\label{3}
(f, g)_{\mu}=\sum_{x\in V}f(x)g(x)\mu(x).
\end{equation}
\end{Def}
With this inner product $C(V)$ becomes a Hilbert space, and we can
write $C(V)=l^2(V,\mu)$. By the definition of the degree and the
symmetry of the weight function, we can check that
\begin{itemize}
  \item $\mu$ is invariant w.r.t. $\{m_x(\cdot)\}$, i.e. $\sum_{x\in V}m_x(y)\mu(x)=\mu(y)$, $\forall y\in V$;
  \item $\mu$ is reversible w.r.t. $\{m_x(\cdot)\}$, i.e. $m_x(y)\mu(x)=m_y(x)\mu(y)$, $\forall x, y\in V$.
\end{itemize}
These two facts imply immediately that the operator $\Delta$ is
nonpositive and self-adjoint on the space $l^2(V,\mu)$.
We call $\lambda$ an eigenvalue of $\Delta$ if there exists some
$f\not\equiv 0$ such that $\Delta f=-\lambda f$. Using this
convention it follows from the observation that $\Delta$ is
self-adjoint and nonpositive that all its eigenvalues are real and
nonnegative. In fact, it's well known that (see e.g. Chung
\cite{Chung})
$0=\lambda_0\leq\lambda_1\leq\cdots\leq\lambda_{N-1}\leq 2$. Since
our graph is connected we actually have $0<\lambda_1$. In Chung
\cite{Chung} it is shown, by proving a discrete version of the
Cheeger inequality, that $\lambda_1$ is a measure for how
easy/difficult it is to cut the graph into two large pieces.
Furthermore, it is well known that $\lambda_{N-1} = 2$ if and only
if $G$ is bipartite. In Bauer-Jost \cite{BJ} a Cheeger type
estimate for the largest eigenvalue $\lambda_{N-1}$ was obtained.
The results in Bauer-Jost \cite{BJ} show that $\lambda_{N-1}$ is a
measure for how close (the meaning of close is made precise in
\cite{BJ}) a graph is to a bipartite one. In the following, we
will call $\lambda_1$ the first eigenvalue and $\lambda_{N-1}$ the
largest eigenvalue of the operator $\Delta$.
\subsection{Neighborhood graphs}In this section, we discuss the deep relationship between
eigenvalues estimates for the Laplace operator $\Delta$ and random
walks on the graph $G$. In particular, we recall the neighborhood
graph method developed by Bauer-Jost \cite{BJ}.

We first introduce the following notation. For a probability
measure $\mu$, we denote
$$\mu P(\cdot):=\sum_{x}\mu(x)m_{x}(\cdot).$$
Let $\delta_x$ be the Dirac measure at $x$, then we can write
$\delta_x P^1(\cdot):=\delta_x P(\cdot)=m_x(\cdot)$. Therefore the
distribution of a $t$-step random walk starting from $x$ with a
transition probability $m_x$ is
\begin{equation}
\delta_x
P^t(\cdot)=\sum_{x_1,\ldots,x_{t-1}}m_x(x_1)m_{x_1}(x_2)\cdots
m_{x_{t-1}}(\cdot)
\end{equation}for $t>1$.
The idea is now to define a family of graphs $G[t]$, $t\geq1$ that
encodes the transition probabilities of the $t$-step random walks
on the graph $G$.
\begin{Def}The neighborhood graph $G[t]=(V, E[t])$ of the graph $G=(V,E)$ of order $t\geq 1$
has the same vertex set as $G$ and the weights of the edges of
$G[t]$ are defined in terms of the transitions probabilities of
the $t$-step random walk,
\begin{equation}\label{4}
w_{xy}[t] := \delta_xP^t(y)d_x.
\end{equation}
\end{Def}In particular, $G = G[1]$ and $x\sim y$ in $G[t]$ if and only if there
exists a path of length $t$ between $x$ and $y$ in $G$. It is easy
to see that the neighborhood graph $G[t]$ is in general a weighted
graph with loops, even if the original graph $G$ is an unweighted,
simple graph. Moreover, we note here that the neighborhood graph
method is related to the discrete heat kernel $p_t(x,y)$ (see e.g.
\cite{ACG} and the references therein) on graphs by $$p_t(x,y) =
\frac{w_{xy}[t]}{d_xd_y}.$$


\begin{Lem}[Bauer-Jost \cite{BJ}]\label{L4}The neighborhood graph $G[t]$ has the
following properties:

\begin{itemize}
\item[(i)]If $t$ is even, then $G[t]$ is connected if and only if
$G$ is not bipartite. Furthermore, if $t$ is even, $G[t]$ is not
bipartite. \item[(ii)]If $t$ is odd, then $G[t]$ is always
connected and $G[t]$ is bipartite iff $G$ is bipartite.
\item[(iii)] $d_x[t] = d_x$ for all $x\in V$.
\end{itemize}
\end{Lem}
We mention the following crucial observation which can for
instance be found in \cite{BJ}:
\begin{Obs}\label{T4} The Laplace operator $\Delta$ on $G$ and
the Laplace operator $\Delta[t]$ on $G[t]$ are related to each
other by the following identity:
\[\Delta[t]= -\mathrm{id} + (\mathrm{id} + \Delta)^t.\]
\end{Obs}
An easy consequence of this observation is that the eigenvalues of
$\Delta[t]$ satisfy
\begin{equation}\label{lessthanone}0=\lambda_0[t]\leq \lambda_1[t]\leq\ldots\leq
\lambda_{N-1}[t]\leq 1\end{equation} if $t$ is even. The reason
why all eigenvalues of $G[t]$ (for $t$ even) are less or equal to
one is that every vertex in the neighborhood graph $G[t]$ has a
loop. Because of equation \eqref{lessthanone} we can assume in the
following that bounds for the eigenvalues of $\Delta[t]$, $t$
even,  are less or equal to one.

The importance of Observation \ref{T4} comes from the following
corollary that establishes a connection between estimates for the
smallest and the largest eigenvalue on $G$ and $G[t]$,
respectively.

\begin{Cor}[Bauer-Jost \cite{BJ}]\label{T3}
\begin{itemize} \item[(i)]  Let $\mathcal{A}[t]$ be a lower bound for the eigenvalue
$\lambda_1[t]$ of $\Delta[t]$, i.e., $\lambda_1[t]\geq
\mathcal{A}[t]$. Then
\begin{equation}\label{7}
1-(1-\mathcal{A}[t])^{\frac{1}{t}}\leq\lambda_1\leq\cdots\leq\lambda_{N-1}\leq
1+(1-\mathcal{A}[t])^{\frac{1}{t}}
\end{equation}
if $t$ is even and
\begin{equation}
1-(1-\mathcal{A}[t])^{\frac{1}{t}}\leq\lambda_1
\end{equation} if $t$ is odd.
\item[(ii)] Let $\mathcal{B}[t]$ be an upper bound for the largest
eigenvalue $\lambda_{N-1}[t]$ of $\Delta[t]$, i.e.
$\lambda_{N-1}[t]\leq \mathcal{B}[t]$. Then all eigenvalues of
$\Delta$ are contained in the union of the intervals
$$\left[0, 1-(1-\mathcal{B}[t])^\frac{1}{t}\right] \bigcup
\left[1+(1-\mathcal{B}[t])^\frac{1}{t}, 2\right]$$ if $t$ is even
and \
$$\lambda_{N-1} \leq 1-(1-\mathcal{B}[t])^\frac{1}{t}$$ if $t$ is odd.
\end{itemize}
\end{Cor}

These results indicate the deep connection between random walks on
graphs and eigenvalue estimate of the Laplace operator. In the
rest of this paper we will use these insights to derive lower
bounds for $\lambda_1$ and upper bounds for $\lambda_{N-1}$ in
terms of the Ollivier-Ricci curvature of a graph.

\subsection{Ollivier-Ricci curvature from a
probabilistic view}\label{ss1} We consider the usual graph metric
$d:V\times V \to \mathbb{R}^+$ on the set of vertices $V$, i.e.
for two distinct points $x, y\in V$, $d(x, y)$ is the number of
edges in the shortest path connecting $x$ and $y$. Then, including
the family of probability measures $m:=\{m_x(\cdot)\}$, we have a
structure $(V, d, m)$, on which the definition of Ricci curvature
proposed by Ollivier \cite{Oll} can be stated.

\begin{Def}[Ollivier \cite{Oll}]
For any two distinct points $x, y\in V$, the (Ollivier-) Ricci
curvature of $(V, d, m)$ along $(xy)$ is defined as
\begin{equation}\label{8}
\kappa (x, y):=1-\frac{W_1(m_x, m_y)}{d(x, y)}.
\end{equation}
\end{Def}
Here, $W_1(m_x, m_y)$ is the transportation distance between the
two probability measures $m_x$ and $m_y$, in a formula,
\begin{equation}\label{9}
W_1(m_x, m_y)=\inf_{\xi^{x,y}\in \prod(m_x, m_y)}\sum_{(x',y')\in
V\times V}d(x', y')\xi^{x,y}(x', y'),
\end{equation}
where $\prod(m_x, m_y)$ is the set of probability measures
$\xi^{x,y}$ that satisfy
\begin{equation}\label{9'} \sum_{y'\in V}\xi^{x,y}(x', y')=m_x(x'), \,\,\sum_{x'\in V}\xi^{x,y}(x', y')=m_y(y').
\end{equation}
The conditions (\ref{9'}) simply ensure that we start with the
measure $m_x$ and end up with $m_y$. Intuitively, $W_1(m_x, m_y)$
is the minimal cost to transport the mass of $m_x$ to that of
$m_y$ with the distance as the cost function. We also call such a
$\xi^{x,y}$ a transfer plan between $m_x$ and $m_y$, or a coupling
of two random walks governed by $m_x$ and $m_y$, respectively.
Those $\xi^{x,y}$ ($\xi^{x,y}$ might not be unique) which attain
the infimum value in (\ref{9}), are called optimal couplings. The
optimal coupling exists in a very general setting. For  locally
finite graphs the existence follows from a simple and interesting
argument in Remark 14.2 in \cite{LPW}.

By the Kantorovich duality formula for transportation distances
$W_1(m_x,m_y)$ is also given in the form,
\begin{equation}\label{10}
W_1(m_x, m_y)=\sup_{f:\mathrm{Lip}(f)\leq 1}\left[\sum_{x'\in
V}f(x')m_x(x')-\sum_{y'\in V}f(y')m_y(y')\right],
\end{equation}
where $\mathrm{Lip}(f):=\sup_{x\neq y}\frac{|f(x)-f(y)|}{d(x,
y)}.$ For more details about this concept, we refer to Villani
\cite{V1, V2}, and Evans \cite{Evans}.

For the rest of this paper, let $k$ be a lower bound for the
Ollivier-Ricci curvature, i.e.
\begin{equation}\label{11}
\kappa(x, y)\geq k, \,\,\forall x\sim y.
\end{equation}
The Ricci curvature satisfies the following properties (see
\cite{Oll}):
\begin{Lem}\begin{itemize}
\item[$(i)$]If $\kappa(x,y)\geq k$ for all neighbors $x\sim y$,
then this is true for all pairs of vertices $x,y\in V$.
\item[$(ii)$]We have $-2\leq \kappa(x,y)\leq 1.$
\end{itemize}
\end{Lem}
%

We will derive more precise lower and upper bounds for $\kappa$ on
a locally finite graph with loops in Theorem \ref{T6} and Theorem
\ref{T7}, respectively (see also Lin-Yau \cite{LinYau} and
Jost-Liu \cite{JL} for related results).

Combining \eqref{8} and \eqref{11} we obtain
\begin{equation}\label{12}
W_1(m_x, m_y)\leq (1-k)d(x, y)=1-k, \,\,\forall x\sim y,
\end{equation}
which is essentially equivalent to the well known path coupling
criterion on the state space of Markov chains used to study the
mixing time of them (see \cite{BD,LPW,Pe}). We will utilize this
idea to interpret the lower bound of the Ollivier-Ricci curvature
as a control on the expectation value of the distance between two
coupled random walks.

We reformulate Bubley-Dyer's theorem (see \cite{BD} or \cite{LPW},
\cite{Pe}) in our language.
\begin{Thm}[Bubley-Dyer]\label{BD} On $(V, d, m)$, if for each pair of neighbors $x, y\in V$, we have the contraction
$$W_1(m_x, m_y)\leq (1-k)d(x, y) =1-k ,$$
then for any two probability measures $\mu$ and $\nu$ on $V$, we
have
$$W_1(\mu P, \nu P)\leq (1-k)W_1(\mu, \nu).$$
\end{Thm}
With this at hand, it is easy to see that if for any pair of
neighbors $x, y$, $\kappa(x, y)\geq k$, then for any time $t$ and
any two $\bar{x}, \bar{y}$, which are not necessarily neighbors,
the following is true,
\begin{equation}\label{14}
W_1(\delta_{\bar{x}} P^t, \delta_{\bar{y}} P^t)\leq
(1-k)^td(\bar{x}, \bar{y}).
\end{equation}
We consider two coupled discrete time random walks $(\bar{X}_t,
\bar{Y}_t)$, whose distributions are $\delta_{\bar{x}} P^t$,
$\delta_{\bar{y}} P^t$ respectively. They are coupled in a way
that the probability
$$p(\bar{X}_t=\bar{x}', \bar{Y}_t=\bar{y}')=\xi_t^{\bar{x}, \bar{y}}(\bar{x}', \bar{y}'),$$
where $\xi_t^{\bar{x}, \bar{y}}(\cdot, \cdot)$ is the optimal
coupling of $\delta_{\bar{x}} P^t$ and $\delta_{\bar{y}} P^t$. In
this language, we can interpret the term $W_1(\delta_{\bar{x}}
P^t, \delta_{\bar{y}} P^t)$ as the expectation value of the
distance ${\bf E}^{\bar{x}, \bar{y}}d(\bar{X}_t, \bar{Y}_t)$
between the coupled random walks $\bar{X}_t$ and $\bar{Y}_t$.
\begin{Cor}\label{C1}
On $(V, d, m)$, if $\kappa(x, y)\geq k, \,\,\forall x\sim y$, then
we have for any two $\bar{x}, \bar{y}\in V$,
\begin{equation}\label{15}
{\bf E}^{\bar{x}, \bar{y}}d(\bar{X}_t,
\bar{Y}_t)=W_1(\delta_{\bar{x}} P^t, \delta_{\bar{y}} P^t)\leq
(1-k)^td(\bar{x}, \bar{y}).
\end{equation}
\end{Cor}

\section{Estimates for Ollivier-Ricci curvature on locally finite graphs with loops}\label{Section3}
In \cite{JL} Jost-Liu obtained a sharp estimate for Ollivier-Ricci
curvature on locally finite graphs without loops. As mentioned
above, neighborhood graphs are in general weighted graphs with
loops. Therefore, for our purposes, we need to understand the
curvature of graphs with loops.  In this section, we generalize
the estimates in Jost-Liu \cite{JL} for locally finite graphs
$\tilde{G}=(\tilde{V}, \tilde{E})$ that may have loops. This is
done by considering a novel optimal transportation plan.

We first fix some notations. For any two real numbers $a, b$,
$$a_+:=\max\{a, 0\}, \,\,a\wedge b:=\min\{a, b\}, \,\text{and} \,\,a\vee b:=\max\{a, b\}.$$
We denote $\tilde{N}_x:=\{z\in \tilde{V}|z\sim x\}$ as the
neighborhood of $x$ and $N_x:=\tilde{N}_x\cup\{x\}$. Then
$N_x=\tilde{N}_x$ if $x$ has a loop. For every pair of neighbors
$x, y$, we divide $N_x, N_y$ into disjoint parts as follows.
\begin{equation} \label{Eq3}N_x=\{x\}\cup \{y\}\cup N_x^1\cup N_{xy},
\,\,N_y=\{y\}\cup\{x\}\cup N_y^1\cup N_{xy},
\end{equation} where $$N_{xy}=N_{x\geq y}\cup N_{x<y}$$ and
\begin{align*}
&N_x^1:=\{z|z\sim x, z\not\sim y,  z\neq y\},\,  \\
&N_{x\geq y}:=\{z|z\sim x, z\sim y, z\neq x, z\neq y,
\frac{w_{xz}}{d_x}\geq \frac{w_{zy}}{d_y}\},\\& N_{x<y}:=\{z|z\sim
x, z\sim y, z\neq x, z\neq y,
\frac{w_{xz}}{d_x}<\frac{w_{zy}}{d_y}\}.
\end{align*}In the next figure we illustrate this partition of the vertex set.
\begin{center}
\includegraphics[width = 8 cm]{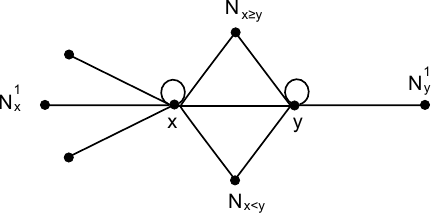}
\end{center}

\begin{Thm}\label{T6}
On $\tilde{G}=(\tilde{V}, \tilde{E})$, we have for any pair of
neighbors $x, y\in \tilde{V}$,
\begin{align*}
\kappa(x, y)\geq k(x, y):=&-\left(1-\frac{w_{xy}}{d_x}-\frac{w_{xy}}{d_y}-\sum_{x_1\in N_{xy}}\frac{w_{x_1x}}{d_x}\vee\frac{w_{x_1y}}{d_y}\right)_+\\
&-\left(1-\frac{w_{xy}}{d_x}-\frac{w_{xy}}{d_y}-\sum_{x_1\in
N_{xy}}\frac{w_{x_1x}}{d_x}\wedge\frac{w_{x_1y}}{d_y}\right)_+
\\&+\sum_{x_1\in N_{xy}}\frac{w_{x_1x}}{d_x}\wedge\frac{w_{x_1y}}{d_y}+\frac{w_{xx}}{d_x}+\frac{w_{yy}}{d_y}.
\end{align*}
Moreover, this inequality is sharp.
\end{Thm}
\begin{Rmk}\label{Re2}
On an unweighted graph, the lower bound for the Ricci curvature
$k(x,y)$ for $x\sim y$ becomes
\begin{align*}k(x, y)=&-\left(1-\frac{1}{d_x}-\frac{1}{d_y}-\frac{\sharp(x, y)}{d_x\wedge d_y}\right)_+
-\left(1-\frac{1}{d_x}-\frac{1}{d_y}-\frac{\sharp(x, y)}{d_x\vee
d_y}\right)_+\\&+\frac{\sharp(x, y)}{d_x\vee
d_y}+\frac{c(x)}{d_x}+\frac{c(y)}{d_y},\end{align*} where
$\sharp(x, y):=\sum_{x_1\in N_{xy}}1$ is the number of triangles
containing $x,y$, $c(x)=0$ or $1$ is the number of loops at $x$.
\end{Rmk}
\begin{proof} Since the total  mass of $m_x$ is equal to one, we
obtain from (\ref{Eq3}) the following identity for neighboring
vertices $x$ and $y$: \begin{equation} \label{Eq4}1 -
\frac{w_{xy}}{d_x} - \sum_{x_1\in N_{xy}} \frac{w_{x_1x}}{d_x} =
\frac{w_{xx}}{d_x} +\sum_{x_1\in
N_x^1}\frac{w_{x_1x}}{d_x}\end{equation} A similar identity holds
for $y$.

We denote
\begin{align*}
A_{x, y}&:=1-\frac{w_{xy}}{d_x}-\frac{w_{xy}}{d_y}-\sum_{x_1\in N_{xy}}\frac{w_{x_1x}}{d_x}\vee\frac{w_{x_1y}}{d_y}, \\
B_{x, y}&:=1-\frac{w_{xy}}{d_x}-\frac{w_{xy}}{d_y}-\sum_{x_1\in
N_{xy}}\frac{w_{x_1x}}{d_x}\wedge\frac{w_{x_1y}}{d_y}.
\end{align*}
Obviously, $A_{x, y}\leq B_{x, y}$. We firstly try to understand
these two quantities.

If $A_{x, y}\geq 0$, we have
\begin{equation}\label{31}
1-\frac{w_{xy}}{d_y}-\sum_{x_1\in
N_{xy}}\frac{w_{x_1y}}{d_y}\geq\frac{w_{xy}}{d_x}+\sum_{x_1\in
N_{x\geq
y}}\left(\frac{w_{xx_1}}{d_x}-\frac{w_{x_1y}}{d_y}\right),
\end{equation}
i.e., using (\ref{Eq4}) we observe that the mass of $m_y$ at $y$
and $N_y^1$ is no smaller than that of $m_x$ at $y$ and the excess
mass at $N_{x\geq y}$. Rewriting (\ref{31}) in the form
\begin{equation*}
\frac{w_{xy}}{d_y}+\sum_{x_1\in N_{xy}}\frac{w_{x_1y}}{d_y}\leq
1-\frac{w_{xy}}{d_x}-\sum_{x_1\in N_{x\geq
y}}\left(\frac{w_{xx_1}}{d_x}-\frac{w_{x_1y}}{d_y}\right),
\end{equation*}and
subtracting the term $\sum_{x_1\in
N_{xy}}\frac{w_{xx_1}}{d_x}\wedge\frac{w_{x_1y}}{d_y}$ on both
sides we obtain
\begin{equation}\label{32}
\frac{w_{xy}}{d_y}+\sum_{x_1\in
N_{x<y}}\left(\frac{w_{x_1y}}{d_y}-\frac{w_{xx_1}}{d_x}\right)\leq1-\frac{w_{xy}}{d_x}-\sum_{x_1\in
N_{xy}}\frac{w_{xx_1}}{d_x},
\end{equation}
i.e., the mass of $m_x$ at $x$ and $N_x^1$ is larger than that of
$m_y$ at $x$ and the excess mass at $N_{x<y}$.

If $B_{x, y}\geq 0$, we have
\begin{equation}\label{33}
1-\frac{w_{xy}}{d_x}-\sum_{x_1\in
N_{xy}}\frac{w_{xx_1}}{d_x}+\sum_{x_1\in N_{x\geq
y}}\left(\frac{w_{xx_1}}{d_x}-\frac{w_{x_1y}}{d_y}\right)\geq\frac{w_{xy}}{d_y},
\end{equation}
i.e., the mass of $m_x$ at $x$ and $N_x^1$ and the excess mass at
$N_{x\geq y}$ is no smaller than that of $m_y$ at $x$.

In Jost-Liu \cite{JL} it is explicitly described how much mass has
to be moved from a vertex in $N_x$ to which point in $N_y$, i.e.
the exact value of $\xi^{x, y}(x', y')$, for any $x'\in N_x$,
$y'\in N_y$. But in the case with loops it would be too
complicated if we try to do the same thing. Instead, we adopt here
a dynamic strategy. That is, we think of a discrete time flow of
mass. After one unit time, the mass flows forward for distance $1$
or stays there. We only need to determine the direction of the
flow according to different cases.

As in Jost-Liu \cite{JL}, we divide the discussion into 3 cases.
\begin{itemize}
\item $0\leq A_{x, y}\leq B_{x, y}.$ In this case we use the
following transport plan: Suppose the initial time is $t=0$.
\begin{description}
\item[$t=1$] Move all the mass at $N_x^1$ to $x$ and the excess
mass at $N_{x\geq y}$ to $y$. We denote the distribution of the
mass after the first time step by $m^1$. We have
$$W_1(m_x,m^1)\leq \left(1-\frac{w_{xx}}{d_x}-\frac{w_{xy}}{d_x}-\sum_{x_1\in N_{xy}}
\frac{w_{xx_1}}{d_x}\right)\times 1+\sum_{x_1\in N_{x\geq
y}}\left(\frac{w_{xx_1}}{d_x}-\frac{w_{x_1y}}{d_y}\right)\times 1
$$ \item[$t=2$] Move one part of the excess mass at $x$ now to
fill the gap at $N_{x<y}$ and the other part to $y$. By (\ref{32})
the mass at $x$ after $t=1$ is enough to do so. The distribution
of the mass is now denoted by $m^2$. We have \begin{align*}
W_1(m^1,m^2)\leq & \sum_{x_1\in
N_{x<y}}\left(\frac{w_{x_1y}}{d_y}-\frac{w_{xx_1}}{d_x}\right)\times
1+ \\&\left[\left(1-\frac{w_{xy}}{d_x}-\sum_{x_1\in
N_{xy}}\frac{w_{xx_1}}{d_x}\right)-\sum_{x_1\in
N_{x<y}}\left(\frac{w_{x_1y}}{d_y}-\frac{w_{xx_1}}{d_x}\right)-\frac{w_{xy}}{d_y}\right]\times
1\end{align*} \item[$t=3$] Move the excess mass at $y$ now to
$N_y^1$. We denote the mass after the third time step by
$m^3=m_y$. We have
\begin{align*}W_1(m^2,m_y)\leq &\Bigg[\left(1-\frac{w_{xy}}{d_x}-
\sum_{x_1\in N_{xy}}\frac{w_{xx_1}}{d_x}\right)-\sum_{x_1\in
N_{x<y}}
\left(\frac{w_{x_1y}}{d_y}-\frac{w_{xx_1}}{d_x}\right)-\frac{w_{xy}}{d_y}\\
&+\frac{w_{xy}}{d_x}+\sum_{x_1\in N_{x\geq
y}}\left(\frac{w_{xx_1}}{d_x}-\frac{w_{x_1y}}{d_y}\right)-\frac{w_{yy}}{d_y}\Bigg]\times
1\end{align*}
\end{description}
By triangle inequality and (\ref{9}), we get
\begin{align*}
W_1(m_x, m_y)\leq& W_1(m_x, m^1)+W_1(m^1, m^2)+W_1(m^2, m_y)
\\
=&3-2\frac{w_{xy}}{d_x}-2\frac{w_{xy}}{d_y}-2\sum_{x_1\in
N_{xy}}\frac{w_{xx_1}}{d_x}\wedge\frac{w_{x_1y}}{d_y}-\sum_{x_1\in
N_{xy}}\frac{w_{xx_1}}{d_x}\vee\frac{w_{x_1y}}{d_y}\\&-\frac{w_{xx}}{d_x}-\frac{w_{yy}}{d_y}.
\end{align*}
Moreover, if the following function can be extended as a function
on the graph such that $\mathrm{Lip}(f)\leq 1$, (i.e., if there
are no paths of length 1 between $N_x^1$ and $N_{x<y}$, nor paths
of length 1 between $N_y^1$ and $N_{x\geq y}$, nor paths of length
$1$ or $2$ between $N_x^1$ and $N_y^1$,)
\begin{equation*}
 f(z)=\left\{
        \begin{array}{ll}
          0, & \hbox{if $z\in N_y^1$;} \\
          1, & \hbox{if $z\in\{y\}\cup N_{x<y}$;} \\
          2, & \hbox{if $z\in\{x\}\cup N_{x\geq y}$;} \\
          3, & \hbox{if $z\in N_x^1$,}
        \end{array}
      \right.
\end{equation*}
then by Kantorovich duality (\ref{10}), we can show that the
inequality above is actually an equality. Recalling the definition
of $\kappa(x, y)$, we have proved the theorem in this case.
  \item$A_{x, y}<0\leq B_{x, y}.$ We use the following transfer plan:
  \begin{description}
    \item[$t=1$] We divide the excess mass of $m_x$ at $N_{x\geq y}$
into two parts. One part together with the mass of $m_x$ at $y$ is
enough to fill gaps at $y$ and $N_y^1$. Since (\ref{31}) doesn't
hold in this case, this is possible. We move this part of mass to
$y$ and the other part to $x$. We also move all the mass of $m_x$
at $N_{x}^1$ to x.
    \item[$t=2$] We move the excess mass at $x$ now to $N_{x<y}$ and the excess mass at $y$ to $N_{y}^1$.
  \end{description}
Applying this transfer plan, we can prove (we omit the calculation
here) $$W_1(m_x, m_y)\leq
2-\frac{w_{xy}}{d_x}-\frac{w_{xy}}{d_y}-2\sum_{x_1\in
N_{xy}}\left(\frac{w_{xx_1}}{d_x}\wedge\frac{w_{x_1y}}{d_y}\right)-\frac{w_{xx}}{d_x}-\frac{w_{yy}}{d_y}.$$
Moreover, if the following function can be extended as a function
on the graph such that $\mathrm{Lip}(f)\leq 1$, (i.e., if there
are no paths of length 1 between $N_x^1\cup N_{x\geq y}$ and
$N_y^1\cup N_{x<y}$,)
\begin{equation*}
 f(z)=\left\{
        \begin{array}{ll}
          0, & \hbox{if $z\in N_y^1\cup N_{x<y}$;} \\
          1, & \hbox{if $z=x$ or $z=y$;} \\
          2, & \hbox{if $z\in N_{x}^1\cup N_{x\geq y}$,} \\
        \end{array}
      \right.
\end{equation*}
then by Kantorovich duality (\ref{10}), we can check that the
inequality above is actually an equality.
  \item $A_{x, y}\leq B_{x, y}<0.$ We use the following transport
  plan:
  \begin{description}
    \item[$t=1$] Move the mass of $m_x$ at $N_x^1$ and $N_{x\geq y}$ to $x$. Since now (\ref{33}) doesn't hold, we need to move one part of the mass $m_y(y)$ to $x$ and the other part to $N_y^1$ and $N_{x<y}$.
  \end{description}
  Applying this transfer plan, we can calculate
  $$W_1(m_x, m_y)\leq 1-\sum_{x_1\in N_{xy}}\left(\frac{w_{xx_1}}{d_x}\wedge\frac{w_{x_1y}}{d_y}\right)-\frac{w_{xx}}{d_x}-\frac{w_{yy}}{d_y}.$$
  Since the following function can be extended as a function on the graph such that $\mathrm{Lip}(f)\leq 1$,
  \begin{equation*}
 f(z)=\left\{
        \begin{array}{ll}
          0, & \hbox{if $z\in\{x\}\cup N_{x<y}\cup N_y^1$;} \\
          1, & \hbox{if $z\in\{y\}\cup N_{x\geq y} \cup N_x^1$,} \\
        \end{array}
      \right.
\end{equation*}
we can check the inequality above is in fact an equality by
Kantorovich duality. That is, in this case for any $x\sim y$,
$$\kappa(x, y)=\sum_{x_1\in N_{xy}}\left(\frac{w_{xx_1}}{d_x}\wedge\frac{w_{x_1y}}{d_y}\right)+\frac{w_{xx}}{d_x}+\frac{w_{yy}}{d_y}.$$
\end{itemize}\end{proof}

We also have a generalization of the upper bound in Jost-Liu
\cite{JL} on $\tilde{G}$.
\begin{Thm}\label{T7}
On $\tilde{G}=(\tilde{V}, \tilde{E})$, we have for every pair of
neighbors $x, y$,
\begin{equation*}
 \kappa(x, y)\leq \sum_{x_1\in\{x\}\cup\{y\}\cup N_{xy}}\frac{w_{x_1x}}{d_x}\wedge\frac{w_{x_1y}}{d_y}.
\end{equation*}
\end{Thm}
\begin{proof} $I:=\sum_{x_1\in\{x\}\cup\{y\}\cup
N_{xy}}\frac{w_{x_1x}}{d_x}\wedge\frac{w_{x_1y}}{d_y}$ is exactly
the mass of $m_x$ which we need not move. The other mass need to
be moved for at least distance 1. So we have $W_1(m_x, m_y)\geq
1-I$, which implies $\kappa(x, y)\leq I$, for $x\sim y$.
\end{proof}

\begin{Ex}\label{Ex2}
We consider a lazy random walk on an unweighted complete graph
$\mathcal{K}_N$ with $N$ vertices governed by $m_x(y)=1/N, \forall
x, y$. Or equivalently , we consider the graph
$\mathcal{K}^{\mathrm{lazy}}_N$. Using Theorem \ref{T6} and
Theorem \ref{T7}, we get for any $x, y$
$$1=\frac{N-2}{N}+\frac{1}{N}+\frac{1}{N}\leq \kappa(x, y)\leq \frac{1}{N}\cdot N=1.$$
That is, in this case, both the lower and the upper bound are
sharp.
\end{Ex}

An immediate consequence of Theorem \ref{T7} is the following
important observation.
\begin{Cor}\label{Re1}
If there exists two vertices $x\sim y$ in $G$ such that $\sharp(x,
y)= c(x)=c(y)= 0$ then $\kappa(x,y)\leq 0$ and hence $k\leq 0$.
\end{Cor}

This corollary shows that positive Ricci curvature is a quite
strong requirement. For instance, in a loopless graph, already the
existence of a single edge that is not contained in a triangle
prevents the graph from having a positive Ricci curvature lower
bound. We will show in the following that the neighborhood graph
technique can be used to reduce the influence of such edges. This
observation is particularly important in the next section when we
study eigenvalue estimates in terms of the Ricci curvature.

Neighborhood graphs are nothing but coarse representations of the
original graph.  More precisely, the neighborhood graphs $G[t]$
encode the larger scale structure of the original graph $G$, where
larger values of $t$ stand for larger scales, in the sense that an
edge between two nodes in $G[t]$ is equivalent to the existence of
a path of length $t$ in the original graph $G$ between these two
nodes. In order to see how neighborhood graphs can reduce the
influence of single edges, we state the following simple
observations that follow immediately from the definition of the
neighborhood graphs.

\begin{Obs}\label{Obs}
\begin{itemize}
\item Triangles and loops are preserved when we go to higher order
neighborhood graphs, i.e. if $(xyz)$ form a triangle in $G[s]$
($x$ has a loop in $G[s]$) then they from a triangle in $G[t]$
($x$ has a loop in $G[t]$) for all $t>s$. \item If $t$ is even,
every vertex has a loop in $G[t]$. \item If $t$ is odd, the edge
set of $G$ is a subset of the edge set of $G[t]$, i.e. $E\subseteq
E[t]$.\item If in $G$ a vertex $x$ is not contained in a triangle
but contained in a cycle of length $3t$ then $x$ is contained in a
triangle in $G[t]$. \item If in $G$ a vertex $x$ is contained in a
cycle of odd length $2l+1$, $l\geq 1$, then $x$ is contained in a
triangle (in fact in a complete graph $\mathcal{K}_{2l+1}$) in
$G[t]$ if $t\geq 2l-1$. \item If in $G$ a vertex  $x$ is not
contained in a triangle but $x\sim y$ where $y$ is contained in a
triangle, then $x$ is also contained in a triangle in $G[t]$ for
all $t\geq 2$.
\end{itemize}
\end{Obs}
These observations show, unless $G$ is bipartite, that the number
of triangles and loops will monotonically increase when we go from
$G$ to $G[t]$. Hence even though the Ricci curvature of the
original graph is negative, Corollary \ref{Re1} does not exclude
that the Ricci curvature of the neighborhood graph $G[t]$ is
positive. Indeed we will show in Theorem \ref{T10} that for all
graphs that are not bipartite there exists a $t'\in \mathbb{N}$
such that the Ricci curvature of the neighborhood graph $G[t]$
satisfies $k[t]:=\min_{x,y}\kappa[t](x,y)>0$ for all $t>t'$.

\section{Estimates of the spectrum in terms of Ollivier-Ricci curvature}\label{S3}
In this section, we obtain nontrivial estimates for the extremal
eigenvalues of the normalized Laplace operator in terms of the
Olliver-Ricci curvature of the neighborhood graphs. In particular,
our new estimates improve the eigenvalue estimates obtained by
Olliver in \cite{Oll}.

In Proposition 30 of \cite{Oll}, Ollivier proved a spectral radius
estimate which works on a general metric space with random walks.
In particular, on finite graphs, it can be stated as follows.
\begin{Thm}[Ollivier] \label{T1}
On $(V, d, m)$, if $\kappa(x, y)\geq k, \,\,\forall x\sim y$, then
the eigenvalues of the normalized graph Laplace operator $\Delta$
satisfy
$$k \leq \lambda_1\leq \ldots\leq \lambda_{N-1} \leq 2- k.$$
\end{Thm}
The lower bound for $\lambda_1$ is a discrete analogue of the
estimate for the smallest nonzero eigenvalue of the
Laplace-Beltrami operator on a Riemannian manifold by Lichnerowicz
\cite{Lich}. As pointed out by Ollivier \cite{Oll}, this result is
also related to the coupling method for estimates of the first
eigenvalue in the Riemannian setting developed by Chen-Wang
\cite{CW} (which leads to a refinement of the eigenvalue estimate
of Li-Yau \cite{LiYau}), see also the surveys Chen \cite{C1, C2}.
The corresponding result of Corollary \ref{C1} in the smooth case,
i.e., controlling the expectation distance of two coupled Markov
chains in terms of the lower bound of Ricci curvature on a
Riemannian manifold, is a key step in Chen-Wang's method.

A direct proof of Theorem \ref{T1} can be found in \cite{Oll}.
Here for readers' convenience,  we present an analogue of
Chen-Wang's method in the discrete setting, which motivated us to
combine the Ollivier-Ricci curvature and the neighborhood graph
method via random walks. It reflects the deep connection between
eigenvalue estimates and random walks or heat equations.
\begin{proof} We consider the transition probability operator $P:
l^2(V,\mu)\to l^2(V,\mu)$ defined by  $Pf(x):=\sum_y f(y)m_x(y)
=\sum_y f(y)\delta_xP(y) $. Then we have $P^tf(x)=\sum_y
f(y)\delta_x P^t(y).$ We construct a discrete time heat equation,
\begin{equation}\label{16}
 \left\{
   \begin{array}{ll}
     f(x, 0)=f_1(x), \\
     f(x,1)-f(x,0)=\Delta f(x,0), \\
     f(x,2)-f(x,1)=\Delta f(x,1), \\
     \cdots\\
     f(x, t+1)-f(x,t)=\Delta f(x, t),
   \end{array}
 \right.
\end{equation}
where $f_1(x)$ satisfies $\Delta f_1(x)=-\lambda
f_1(x)=Pf_1(x)-f_1(x)$ for $\lambda\neq 0$. Iteratively, one can
find the solution of the above system of equations as
\begin{equation}
 f(x,t)=P^t f_1(x)=(1-\lambda)^t f_1(x).
\end{equation}
We remark here that the solution of the heat equation on a
Riemannian manifold with the eigenfunction as the initial value is
$f(x, t)=f_1(x)e^{-\lambda t}$, which also involves information
about both the eigenvalue $\lambda$ and the eigenfunction
$f_1(x)$.

Then we have for any $\bar{x}, \bar{y}\in V$
\begin{align*}
 |1-\lambda|^t|f_1(\bar{x})-f_1(\bar{y})|&=|f(\bar{x},t)-f(\bar{y},t)|
=|P^tf_1(\bar{x})-P^tf_1(\bar{y})|\\
&\leq \sum_{\bar{x}', \bar{y}'}|f(\bar{x}')-f(\bar{y}')|\xi^{\bar{x}, \bar{y}}_t(\bar{x}', \bar{y}')\\
&\leq \mathrm{Lip}(f_1){\bf E}^{\bar{x}, \bar{y}}d(\bar{X}_t,
\bar{Y}_t) \leq \mathrm{Lip}(f_1)(1-k)^td(\bar{x},\bar{y}).
\end{align*}
Here, $\mathrm{Lip}(f)$ is always finite since the underlying
space $V$ is a finite set. In the last inequality we used
Corollary \ref{C1}. From an analytic point of view, the above
calculation can be seen as a gradient estimate for the solution of
the heat equation.

Since the eigenfunction $f_1$ for the eigenvalue $\lambda$ is
orthogonal to the constant function, i.e. $(f_1, \mathbf{1})_\mu
=0$, we can always find $x_0,y_0\in V$ such that
$|f_1(x_0)-f_1(y_0)|>0$. It follows that \begin{equation*}
  0< |f_1(x_0)-f_1(y_0)|\le
  \left(\frac{1-k}{|1-\lambda|}\right)^t\mathrm{Lip} (f_1)d(x_0,y_0)
  \text{ for all }t.
\end{equation*}
To prevent a contradiction when $t\rightarrow \infty$, we need
$|1-\lambda|\leq 1-k$ which completes the proof.
\end{proof}

As an immediate consequence of Theorem \ref{T1} and Theorem
\ref{T6} we obtain an estimate for the largest eigenvalue in terms
of the number of triangles and loops in the graph.
\begin{Cor} On $G=(V, E)$, the largest eigenvalue satisfies
$$\lambda_{N-1} \leq 2-  \min_{x\sim y} k(x,y),$$ where $k(x,y)$ is
defined in Theorem \ref{T6}.
\end{Cor}


By considering the graph $\mathcal{K}^{\mathrm{lazy}}_N$ in
Example \ref{Ex2} it is easy to see that  Ollivier's estimates in
Theorem \ref{T1} can be sharp for certain
graphs. However, from Corollary \ref{Re1}  
we know that a positive lower curvature bound is a strong
restriction on a graph. In the open Problem G in \cite{Oll2}
Olliver asks for the possibility to relax this strong assumption.
We will show in the following how to obtain nontrivial estimates
for all graphs that are not bipartite by using the neighborhood
graph technique. This gives an answer to Ollivier's problem in the
finite graph setting.

Before we show how one can improve Olliver's result by using the
neighborhood graph technique, we show how this technique can be
used to obtain upper bounds for $\lambda_{N-1}$ from lower bounds
for $\lambda_{1}$. We do this by carefully comparing the
Olliver-Ricci curvature on a graph $G$ and its neighborhood graphs
$G[t]$.

If we interpret the graph $G=(V, E)$ as a structure $(V, d,
m=\{\delta_x P\})$, then by (\ref{4}) its neighborhood graph $G[t]
= (V,E[t])$ can be considered as a structure $(V, d[t], \{\delta_x
P^t\})$. So the first step should be to estimate the graph metric
on $G$ and the graph metric on $G[t]$ by each other.

\begin{Lem}\label{L1}
For any $x, y\in V$, we have
\begin{equation}\label{18}
\frac{1}{t}d(x, y)\leq d[t](x, y),
\end{equation}where we use the convention that $d[t](x,y) =
\infty$ if $G[t]$ is not connected. By Lemma \ref{L4} this happens
iff $G$ is bipartite and $t$ is even.
\end{Lem}
\begin{proof} If $G[t]$ is not connected, then \eqref{18} is
trivially satisfied. Otherwise, we can find a shortest path
$x_0=x, x_1, \ldots, x_l=y$, between $x$ and $y$ in $G[t]$, i.e.
$l=d[t](x, y)$. For $x_i, x_{i+1}$, $i=0,\ldots, l-1$, by
definition of neighborhood graph, we have $d(x_i, x_{i+1})\leq t$
in $G$. Equivalently,
$$\frac{1}{t}d(x_i, x_{i+1})\leq 1=d[t](x_i, x_{i+1}).$$
Summing  over all $i$, we get
$$\frac{1}{t}\sum_{i=0}^{l-1}d(x_i, x_{i+1})\leq d[t](x, y).$$
Then the triangle inequality of $d$ on $G$ gives (\ref{18}).
\end{proof}
\begin{Rmk}
In fact, when $t$ is larger than the diameter $D$ of the graph
$G$, we have a better estimate
\begin{equation}\label{19}
\frac{1}{t}d(x, y)<\frac{1}{D}d(x, y)\leq 1\leq d[t](x, y).
\end{equation}
\end{Rmk}

\begin{Lem}\label{L2}
If $E\subseteq E[t]$, then $d[t](x, y)\leq d(x, y)$.
\end{Lem}
\begin{proof}  The proof is obvious. \end{proof}

The importance of Lemma \ref{L2} comes from the observation that
when the Ollivier-Ricci curvature of the graph $G$ is positive,
$E\subseteq E[t]$ is satisfied for all $t$ and hence Lemma
\ref{L2} is applicable. This can be seen as follows. Corollary
\ref{Re1} implies that if $k>0$, then for all $(x, y)\in E$ we
have $\sharp(x, y)\neq 0$ or $c(x)\neq 0$ or $c(y)\neq 0$ which in
turn implies (see Observation \ref{Obs}) that $(x, y)\in E[t]$ for
all $t$.

\begin{Lem}\label{L3}
Let $k$ be a lower bound for $\kappa$ on $G$. If $E\subseteq
E[t]$, then the curvature $\kappa[t]$ of the neighborhood graph
$G[t]$ satisfies
\begin{equation}\label{21}
\kappa[t](x, y)\geq 1-t(1-k)^t, \,\,\,\forall x, y\in V.
\end{equation}
\end{Lem}
\begin{proof}
Let $W_1^{d[t]}$, $W_1^d$ indicate the different cost functions
used in these two quantities. By Lemma \ref{L2}, Corollary
\ref{C1} and Lemma \ref{L1}, we get
\begin{align*}
W_1^{d[t]}(\delta_x P^t, \delta_y P^t)&\leq W_1^d(\delta_x P^t,
\delta_y P^t) \leq (1-k)^td(x, y)\leq t(1-k)^td[t](x, y),
\end{align*} where we used in the first inequality that the transportation distance (\ref{9}) is linear
in the graph distance $d(\cdot, \cdot)$. Recalling the definition
of the curvature, we have proved (\ref{21}). \end{proof}

\begin{Rmk}Now we have reached the point where we can give a short
geometric proof of the upper bound of the largest eigenvalue in
Theorem \ref{T1}. First assume that $k>0$. In this case $E\subset
E[t]$ and thus we can use Lemma \ref{L3}. From Lemma \ref{L3} and
$\lambda_{1}\geq k$, we know on $G[t]$,
$$\lambda_1[t]\geq 1-t(1-k)^t.$$
Then by using Corollary \ref{T3} $(i)$, we get for any even number
$t$,
$$\lambda_{N-1}\leq 1+t^{\frac{1}{t}}(1-k).$$
Letting $t\rightarrow +\infty$, we get $\lambda_{N-1}\leq 2-k$. If
we assume that $k\leq 0$, then $\lambda_{N-1} \leq 2-k$ is
trivially satisfied. \end{Rmk}

Using the neighborhood graph technique, we further obtain the
following generalization of Theorem \ref{T1}:
\begin{Thm}\label{T10}
Let $k[t]$ be a lower bound of the  Ollivier-Ricci curvature of
the neighborhood graph $G[t]$. Then for all $t\geq 1$ the
eigenvalues of $\Delta$ on $G$ satisfy
\begin{equation}\label{Generalized Estimate}
1- (1- k[t])^{\frac{1}{t}}\leq \lambda_1 \leq \cdots \leq
\lambda_{N-1}\leq 1+ (1- k[t])^{\frac{1}{t}}.\end{equation}
Moreover, if $G$ is not bipartite, then there exists a $t'\geq 1$
such that for all $t\geq t'$ the eigenvalues of $\Delta$ on $G$
satisfy
$$ 0<1- (1- k[t])^{\frac{1}{t}}\leq \lambda_1 \leq \cdots \leq \lambda_{N-1}\leq  1+ (1-
k[t])^{\frac{1}{t}}<2.$$
\end{Thm}
\begin{Rmk}Olliver-Ricci curvature is not well defined for two vertices which belong to two different connected components. However, by Lemma \ref{L4}, $G[t]$ is
disconnected iff $G$ is bipartite and $t$ is even. In this case we
use the convention in \eqref{Generalized Estimate} that $k[t] =
-\infty$.
\end{Rmk} \begin{proof} Combining Theorem \ref{T1}, and Corollary
\ref{T3} immediately yields (\ref{Generalized Estimate}).

The second part of this Theorem is proved in two steps. In the
first step, we will show that if $G$ is not bipartite then there
exists a $t'$ such that for all $t\geq t'$ the neighborhood graph
$G[t]$ of $G$ satisfies $w_{xy}[t]\neq 0$ for all $x,y\in V$, i.e.
$G[t]$ is a complete graph and each vertex has a loop. In the
second step, we show that any graph that satisfies $w_{xy}\neq 0$
for all $x, y\in V$ has a positive lower curvature bound, i.e.
$k>0$. This then completes the proof.

\textbf{Step 1:} By the definition of the neighborhood graph it is
sufficient to show that for all $t\geq t'$ there exists a path of
length $t$ between any pair of vertices. Since $G$ is not
bipartite it follows from the definition of bipartiteness that
there exists a path of even and a path of odd length between any
pair of vertices in the graph. Given a path of length $L$ between
$x$ and $y$ then we can find a path of length $L+2$ between $x$
and $y$ as follows: We go in $L$ steps from $x$ to $y$ and then
from $y$ to one of its neighbors and then back to $y$. This is a
path of length $L+2$ between $x$ and $y$. Since $G$ is finite, it
follows that there exists a $t'$ such that for every pair of
vertices there exists paths of length $t$ for all $t\geq t'$.

\textbf{Step 2:} Given a graph that satisfies $w_{xy} \neq 0$ for
all $x,y\in V$.

Since each vertex in the graph is a neighbor of all other
vertices, it is clear that we can move the excess mass of $m_x$
for distance $1$ to anywhere. Therefore $$W_1(m_x, m_y)\leq
1-\sum_{x_1\in V}\frac{w_{xx_1}}{d_x}\wedge\frac{w_{x_1y}}{d_y},$$
which implies
$$\kappa(x,y) \geq \sum_{x_1\in
V}\frac{w_{xx_1}}{d_x}\wedge\frac{w_{x_1y}}{d_y}.$$ By Theorem
\ref{T7}, it follows  that the above inequality is in fact an
equality. Hence for all $x, y\in V$, we have
$$\kappa(x, y)=\sum_{x_1\in V}\frac{w_{xx_1}}{d_x}\wedge
\frac{w_{x_1y}}{d_y}\geq N\frac{\min_{x, y}w_{xy}}{\max_x d_x}
\geq\frac{\min_{x, y}w_{xy}}{\max_{x, y}w_{xy}}>0,$$ since the
weights $w_{xy}$ are positive for every pair $(x,y)$ of vertices.

This completes the proof. \end{proof}


\begin{Ex}
We consider the unweighted cycle $\mathcal{C}_5$ consisting of $5$
vertices. The first and largest eigenvalue of $\Delta$ on
$\mathcal{C}_5$ are $\lambda_1=1-\cos\frac{2\pi}{5}\doteq 0.6910$
and $\lambda_4=1-\cos\frac{4\pi}{5}\doteq 1.8090.$ It is easy to
check that the optimal lower bound $k$ for the curvature is $0$.
So in this case Ollivier's estimates in Theorem \ref{T1} only
yield  trivial estimates.

For the neighborhood graphs $\mathcal{C}_5[2], \mathcal{C}_5[3]$,
and  $\mathcal{C}_5[4]$ we find that the optimal lower bound for
the Olliver-Ricci curvature is $k[2]=1/4,k[3]=3/8 $, and $k[4]=
1/2$, respectively. Hence Theorem \ref{T10} yields nontrivial
estimates, even if  the lower bound for the Ricci curvature of the
original graph is zero.
\end{Ex}

From the proof of Theorem \ref{T10} we see that for all graphs
$G$, $k[t]$ eventually becomes positive for sufficiently large
$t$. The next two theorems are concerned with the behavior of
$k[t]$ as $t\to \infty$.
\begin{Thm}\label{Conv}
If $G$ is not bipartite, the limit
$$\lim_{t\to\infty}  \frac{\log(1-k[t])}{t}:=-a$$
exists with $a\in (0, +\infty]$. That means, $k[t]$ behaves like
$1-P(t)e^{-at}$ as $t\to \infty$ where $P(t)$ is a polynomial in
$t$.
\end{Thm}
\begin{proof}
Let $t'$ be as in the the proof of Theorem \ref{T10} and let $s,t
\ge t'$. This immediately implies that
\begin{equation}\label{metrics} d[t+s](x,y) = d[s](x,y)= d[t](x,y)
= 1, \,\,\,\forall x \neq y.\end{equation} Let $x,y$ be the pair
of vertices in $G[t+s]$ which attains
$\min_{x,y}\kappa[t+s](x,y)$, we have
\begin{align*}1-k[t+s] &= W_1^{d[t+s]}(\delta_xP^{t+s},\delta_yP^{t+s})=
W_1^{d[s]}(\delta_xP^{t+s},\delta_yP^{t+s})\\&\le
(1-k[s])W_1^{d[s]}(\delta_xP^{t},\delta_yP^{t})=
(1-k[s])W_1^{d[t]}(\delta_xP^{t},\delta_yP^{t})\\&\le
(1-k[s])(1-k[t]),
\end{align*} where we used several times \eqref{metrics}, in the first inequality Theorem 1 and
in the last inequality \eqref{12} and \eqref{metrics}. It follows
that $\log(1-k[t])$ is subadditive, i.e.
$$\log(1-k[t+s])\le \log(1-k[t])+\log(1-k[s]) \qquad\forall s,t\ge t'.$$
We can suppose that $\log(1-k[t])$ is finite for every $t$.
Otherwise there exists $t_0$ such that $k[t_0]=1$, which implies
$k[t]=1,\,\,\forall t\geq t_0$ and then the limit exists and is
equal to $-\infty$. Therefore we can use an extension of Fekete's
subadditivity Lemma by Hammersley \cite{Hammersley} to conclude
that the limit $-a$ exists and $-\infty\leq -a\leq 0$.
Furthermore, let $\tilde{t}:=t't$, $t=1,2,\ldots$, be a
subsequence. Since $E[\tilde{t}]=E[t']$, we can use Lemma 5 and
obtain
$$(1-k[\tilde{t}])^{\frac{1}{\tilde{t}}}\leq t^{\frac{1}{\tilde{t}}}(1-k[t'])^{\frac{t}{\tilde{t}}}=\left(\frac{\tilde{t}}{t'}\right)^{\frac{1}{\tilde{t}}}(1-k[t'])^{\frac{1}{t'}}.$$
Therefore $\lim_{\tilde{t}\rightarrow
\infty}\log(1-k[\tilde{t}])^{\frac{1}{\tilde{t}}}\leq
\log(1-k[t'])^{\frac{1}{t'}}<0$, which implies $a>0$.
\end{proof}
\begin{Rmk}
If $G$ is unweighted, non-bipartite and has no self-loop, we have
$\frac{N}{N-1}\leq \lambda_{N-1}\leq 1+(1-k[t])^{\frac{1}{t}}$.
Therefore in this case $a\leq \log(N-1)$.
\end{Rmk}
\begin{Thm} If $G$ is not bipartite, we have for $t\ge t'$
$$k[t]\ge 1 - 2N e^{-t(1-\overline{\lambda})}\frac{\max_x
d_x}{\min_xd_x},$$ where $\overline{\lambda} = \max_{i\neq
0}|1-\lambda_i|>0$ and again $t'$ is as in the proof of Theorem
\ref{T10}.
\end{Thm}
\begin{proof}
Let $\pi$ ($\pi(x) = \frac{d_x}{\mathrm{vol}(G)}$) be the
stationary distribution of the random walk. Let $t\ge t'$ and
$x,y$ be the pair of vertices in $G[t]$ which attains
$\min_{x,y}\kappa[t](x,y)$, noting (\ref{metrics}) we have
\begin{align*}
k[t](x,y) &=  1- W_1^{d[t]}(m_x[t],m_y[t])\ge 1 - N
\max_{x,y,z}|m_x[t](z) -m_y[t](z)|
\\&\ge 1 - 2 N
\max_{x,z}|\delta_xP^t(z)-\pi(z)|
\end{align*} where we used in the first inequality that one has to
move at most $N$ times the maximal difference between any two
$m_x$ and $m_y$. Chung \cite{Chung} showed that
$$\frac{\max_{x,z}|\delta_xP^t(z)-\pi(z)|}{\max_x \pi(x)}\leq \max_{x,z}\frac{|\delta_xP^t(z)-\pi(z)|}{\pi(z)}\leq e^{-t(1-\overline{\lambda})}\frac{\mathrm{vol}(G)}{\min_xd_x}.$$
Thus we have
$$\max_{x,z}|\delta_xP^t(z)-\pi(z)|\le e^{-t(1-\overline{\lambda})}\frac{\max_x d_x}{\min_xd_x} $$
which completes the proof.
\end{proof}

\section{Estimates for the largest eigenvalue
in terms of the number of joint neighbors}

In Bauer-Jost \cite{BJ} it is shown that the next lemma is a
simple consequence of Observation \ref{T4}.
\begin{Lem} \label{L5} Let $u$ be an eigenfunction
of $\Delta$ for the eigenvalue $\lambda$. Then,
\begin{equation} \label{highest2}
2-\lambda =\frac{(u, \Delta[2]u)_{\mu}}{(u, \Delta
u)_{\mu}}=\frac{\sum_{x,y} w_{xy}[2](u(x)-u(y))^2}{\sum_{x,y}
w_{xy}(u(x)-u(y))^2}.
\end{equation}
\end{Lem}

Lemma \ref{L5} can be used to derive further estimates for the
largest eigenvalue $\lambda_{N-1}$ from above and below. We
introduce the following notations:
\begin{Def}\label{69}
Let $\tilde{N}_x$ be the neighborhood of vertex $x$ as in Section
\ref{Section3}. The minimal and the maximal number of joint
neighbors of any two neighboring vertices is defined as
$\tilde{\sharp}_1 := \min_{x\sim y}(\sharp(x,y)+c(x)+c(y))$ and
$\tilde{\sharp}_2 := \max_{x\sim y}(\sharp(x,y)+c(x)+c(y))$,
respectively. Furthermore, we define $W := \max_{x,y}w_{xy}$ and
$w :=\min_{x,y; x\sim y}w_{xy}$.
\end{Def}

\begin{Thm}\label{65}We have the following estimates for
$\lambda_{N-1}$:\begin{itemize}\item[$(i)$] If $E(G)\subseteq
E(G[2])$ then
$$\lambda_{N-1} \leq 2 - \frac{w^2}{W}\frac{\tilde{\sharp}_1}{\max_x
d_x}.$$ \item[$(ii)$]If $E(G[2])\subseteq E(G)$ then
$$2 -
\frac{W^2}{w}\frac{\tilde{\sharp}_2}{\min_x d_x}\leq
\lambda_{N-1}$$
\end{itemize}
\end{Thm}
\begin{proof}On the one hand, we observe that if $E(G)\subseteq
E(G[2])$, then for every pair of neighboring vertices $x\sim y$ in
$G$
\begin{eqnarray}\label{64}\frac{w_{xy}[2]}{w_{xy}}=\frac{\sum_{z}\frac{1}{d_{z}}w_{xz}w_{zy}}{w_{xy}}
&\geq& \frac{w^2}{W}\frac{\tilde{\sharp}_1}{\max_x
d_x}.\end{eqnarray} On the other hand if $E(G[2])\subseteq E(G)$
then for every pair of neighboring vertices $x\sim y$ in $G(2)$ we
have
\begin{equation}
\label{63}\frac{w_{xy}[2]}{w_{xy}}=\frac{\sum_{z}\frac{1}{d_{z}}w_{xz}w_{zy}}{w_{xy}}\leq
\frac{W^2}{w}\frac{\tilde{\sharp}_2}{\min_x d_x}.\end{equation}
Substituting the inequalities (\ref{64}) and (\ref{63}) in
equation (\ref{highest2}) completes the proof.
\end{proof}
For unweighted regular graphs, Theorem \ref{65} $(i)$ improves the
estimate $\lambda_{N-1}\leq 2-k$. Since $\lambda_{N-1}\leq 2-k$
trivially holds if $k\leq 0$ we only consider the case when $k>0$
is a lower curvature bound. The discussion after Lemma \ref{L2}
shows that $k>0$ implies that $E(G)\subseteq E(G[2])$ and hence we
can apply Theorem \ref{65} (i) in this case. From Theorem \ref{T7}
it follows that for an unweighted graph
$$\kappa(x,y) \leq \frac{\sharp(x,y)}{d_x\vee d_y}+\frac{c(x)}{d_x}+\frac{c(y)}{d_y}$$ for all pairs
of neighboring vertices $x,y$. In the case of a $d$-regular graph
$G$ this implies that a lower bound $k$ for the Ollivier-Ricci
curvature must satisfy
$$k \leq \frac{\tilde{\sharp}_1}{d}.$$ Hence for an unweighted $d$-regular
graph Theorem \ref{65} implies
$$\lambda_{N-1}\leq 2 - \frac{\tilde{\sharp}_1}{d}\leq 2-k.$$
\section*{Acknowledgements}\scriptsize{
\noindent The research leading to these results has received
funding from the European Research Council under the European
Union's Seventh Framework Programme (FP7/2007-2013) / ERC grant
agreement n$^\circ$ 267087.
\\\\
F.B. and S.L. thank Prof. M. von Renesse for discussions leading
to the proof of Theorem \ref{Conv}.
\\\\
F.B. thanks A+B Bauer for their hospitality during his stay in
Wankheim. }

\end{document}